\documentclass[psamsfonts,12pt]{amsart} 
\usepackage{amssymb}

\usepackage[dvips]{graphicx}

\usepackage{amssymb}
\usepackage{amsmath}
\usepackage{color}
\usepackage{times}
\usepackage[usenames,dvipsnames]{xcolor}

\usepackage{amsmath,amstext,amssymb,amsopn,amsthm}
\usepackage{amsmath,amssymb,amsthm}

\newcommand{\Rd}{{\mathbb{R}^d}}
\newcommand{\rd}{{\mathbb{R}^d}}
\newcommand{\RR}{{\mathbb{R}}}

\newcommand{\R}{\mathbb{R}}

\newcommand{\E}{\mathbb{E}}
\newcommand{\Ex}{{\mathbb{E}^x}}

\newcommand{\px}{{\mathbb{P}^x}}
\newcommand{\p}{{\mathbb{P}}}

\newcommand{\D}{\mathcal{D}}

\newtheorem{theorem}{Theorem}[section]
\newtheorem{lemma}[theorem]{Lemma}
\newtheorem{proposition}[theorem]{Proposition}

\theoremstyle{definition}

\newtheorem{example}[theorem]{Example}

\theoremstyle{remark}
\newtheorem{remark}[theorem]{Remark}
\newtheorem{acknowledgment}{Acknowledgment}

\numberwithin{equation}{section}

\definecolor{tl}{rgb}{0.7,0.1,0.2}

\newcommand{\TL}[1]{{\textcolor{black}{#1}}}

\newcommand{\mmm}[1]{{\textcolor{black}{#1}}}
\newcommand{\MMM}[1]{{\textcolor{black}{#1}}}

\newcommand{\BB}[1]{{\textcolor{black}{#1}}}
\newcommand{\MCUBED}[1]{{\textcolor{black}{#1}}}
\newcommand{\fe}[1]{{\textcolor{black}{#1}}}
\newcommand{\FE}[1]{{\textcolor{black}{#1}}}

\begin{document}

\title[Space-time fractional Dirichlet problems]{Space-time fractional Dirichlet problems}

\author{Boris Baeumer}
\address{Boris Baeumer, Department of Mathematics and Statistics, University of Otago,
Dunedin, New Zealand}
\email{bbaeumer@maths.otago.ac.nz}

\author{Tomasz Luks}
\address{Tomasz Luks, Institut f\"ur Mathematik, Universit\"at Paderborn, Warburger Strasse 100, D-33098 Paderborn, Germany}
\email{tluks@math.uni-paderborn.de}

\author{Mark M. Meerschaert}
\address{Mark M. Meerschaert, Department of Statistics and Probability, Michigan State University}
\email{mcubed@stt.msu.edu}

\begin{abstract}
This paper establishes explicit solutions for fractional diffusion problems on bounded domains.  It also gives stochastic solutions, in terms of Markov processes time-changed by an inverse stable subordinator whose index equals the order of the fractional time derivative.  Some applications are given, to demonstrate how to specify a well-posed Dirichlet problem for space-time fractional diffusions in one or several variables.  \fe{This solves an open problem in numerical analysis.}
\end{abstract}

\maketitle

\section{Introduction}
Fractional derivatives were invented by Leibnitz in 1695 \cite{MainardiWaves}.  Recently they have found new applications in many areas of science and engineering, see for example these books \cite{Herrmann,KRS,mainardi1997fractals,MainardiWaves,FCbook,Metzler2000,Metzler2004a,Podlubny}.  In particular, partial differential equations that employ fractional derivatives in time are used to model sticking and trapping, a kind of memory effect \cite{FADEmix,Zsolution,limitCTRW,hittingTime,FCAAplwe}.  For practical applications, it is often necessary to employ numerical methods to solve these time-fractional partial differential equations.  A variety of effective numerical schemes have been developed to solve fractional partial differential equations on a bounded domain, along with proofs of stability and convergence, see for example \cite{Deng4thOrder,Deng,DFF,SIAMreview2012,Liuf07,Liuq11,CNfde}.  An important open problem in this area is to show that these problems are well-posed, see \fe{discussion in} Defterli et al. \cite{FCAAnonlocal}.

In this paper, we take a step in that direction, by establishing explicit solutions to a broad class of time-fractional Cauchy problems \cite{fracCauchy} $\partial_t^\beta u(x,t)=Lu(x,t); u(0)=f(x)$ on a regular bounded domain $\Omega$ in $d$-dimensional Euclidean space, where $\partial_t^\beta$ is the Caputo fractional derivative of order $0<\beta<1$ \cite{MainardiWaves,FCbook}, and $L$ is the semigroup generator of some Markov process on $\R^d$ \cite{ABHN2,BSW,Pazy}.  In particular, we allow the operator $L$ to be nonlocal in space.  This includes the cases where $L$ is a space-fractional derivative in one dimension \cite{bensonappl}, a tempered fractional derivative \cite{TemperedLM}, the fractional Laplacian in $d\geq 1$ dimensions \cite{ChenSong97}, or a multiscaling fractional derivative in $d>1$ dimensions \cite{multiFADE}.  One important outcome of this research is to describe \fe{the} appropriate version of these nonlocal operators on a bounded domain.

Our method of proof uses a fundamental result \cite[Theorem 3.1]{fracCauchy} from the theory of semigroups, along with some ideas from the theory of Markov processes.  This probabilistic method also establishes stochastic solutions for these equations, i.e., we describe a stochastic process whose probability density functions solve the time-fractional and space-nonlocal diffusion problem on the bounded domain.  This extends the recent work of Chen et al.\ \cite{CMN} where $L$ is the (nonlocal) fractional Laplacian, and Meerschaert et al. \cite{MNV} where $L$ is a (local) diffusion operator.  \fe{However, since we do not assume that $L$ is self-adjoint in this paper, standard spectral theory does not apply, and hence our approach is quite different.}

\section{The generator of a killed Feller process}\label{sec2}
Let $X:=\{X_t\}_{t\geq 0}$ be a Feller process on $\Rd$. That is, for any $x\in \Rd$, we assume that the linear operators defined by $P_tf(x):=\Ex [f(X_t)]$ for all $t\geq 0$ form a strongly continuous, \TL{contraction} semigroup on the Banach space $C_0(\Rd)$ of continuous functions $f:\Rd\to\RR$ endowed with the supremum norm $\|f\|:=\sup\{|f(x)|:x\in \Rd\}$, \fe{so that} $P_tf\in C_0(\Rd)$ for all $f\in C_0(\Rd)$. By strongly continuous we mean that $\|P_t f-f\|\to 0$ as $t\searrow 0$ for all $f\in C_0(\Rd)$, and by \TL{contraction} we mean that $\|f\|\leq 1$ implies $\|P_t f\|\leq 1$ for all $f\in C_0(\Rd)$. Then the infinitesimal generator of $X$ is defined by
\begin{equation}\label{GenDef}
Lf:=\lim_{t\searrow0}\frac{P_tf-f}{t} \quad\text{in $C_0(\Rd)$.}
\end{equation}
We denote by $\D(L)$ the domain of $L$ in $C_0(\Rd)$.  Since $f$ is a function of $x\in\Rd$, we can also write the {\it pointwise formula}
\begin{equation}\label{GenDefPtwise23}
L^\sharp f(x):=\lim_{t\searrow0}\frac{\Ex [f(X_t)]-f(x)}{t} \quad\text{in $\Rd$.}
\end{equation}
Since convergence in $C_0(\Rd)$ implies pointwise convergence in $\Rd$, we have $Lf(x)=L^\sharp f(x)$ for all $f\in \D(L)$ and $x\in\Rd$.
\FE{Conversely, an application of the Maximum Principle \cite[Lemma 1.28]{BSW} shows that, for any Feller semigroup, if \eqref{GenDefPtwise23} holds for each $x\in\Rd$, and if the limit $L^\sharp f\in C_0(\Rd)$, then \eqref{GenDef} also holds \cite[Theorem 1.33]{BSW}.}

\FE{This leads to an explicit pointwise formula for the generator: Let $C^k_0(\R^d)$ denote the set of $f\in C_0(\Rd)$ whose derivatives up to order $k$ also belong to $C_0(\Rd)$, and write $C_c^\infty(\R^d)$ for the functions in $C^\infty_0(\R^d)$ that vanish off a compact set.   If $C_c^\infty(\R^d)\subset \D(L)$,
then \cite[Theorem 2.37]{BSW} shows that for any $f\in  C^2_0(\R^d)$ we have $Lf(x)=L^\sharp f(x)=L^p f(x)$ for every $x\in\Rd$, where the pseudodifferential operator
\begin{equation}\label{PDO}\begin{split}
  L^p f(x):=&
  -c(x)f(x)+l(x)\cdot\nabla f(x)+\nabla\cdot Q(x)\nabla f(x)\\
  &+\int_{\R^d\setminus\{0\}} \left(f(x+y)-f(x)-\nabla f(x)\cdot yI_{B_1}(y)\right)\,N(x,dy)
  \end{split}
\end{equation}
for some $c(x)\ge 0$, $l(x)\in\R^d$, $Q(x)\in \R^{d\times d}$ symmetric and positive definite, $N(x,\cdot)$ a positive measure satisfying $\int_{\R^d\setminus\{0\}}\min(|y|^2,1)N(x,dy)<\infty$, and $B_1$ the unit ball.  The goal of this section is to apply this same procedure to killed Feller processes on a bounded domain.}

\begin{remark}\label{RemBSWcore}
In applications, there are no generally useful sufficient conditions that guarantees $C_c^\infty(\R^d)\subset \D(L)$, so one has to check this on a case-by-case basis, see for example \cite[Chapter 3]{BSW}.  \FE{In the special case of an infinitely divisible L\'evy process $X_t$, where $c=0$ and $l,Q,N$ do not depend on $x\in\R^d$, it follows from Sato \cite[Theorem 31.5]{SAT} that $C_c^\infty(\R^d)\subset C^2_0(\R^d)\subset \D(L)$.  Hence we always have $Lf(x)=L^\sharp f(x)=L^p f(x)$ for all $f\in  C^2_0(\R^d)$ and all $x\in\Rd$ in this case.}
\end{remark}

Let $\Omega\subset\Rd$ be a bounded  domain (connected open set) and let $C_0(\Omega)$ denote the set of continuous real-valued functions on \FE{$\Omega$} that tend to zero as $x\in\Omega$ approaches the boundary.  Then $C_0(\Omega)$ is a Banach space with the supremum norm. {For a Feller process $X_t$ on $\RR^d$} we define the first exit time from $\Omega$ for $X_t$ by
\begin{equation}\label{tauDdef}
\tau_\Omega=\inf\left\{t>0:X_t\notin \Omega\right\} .
\end{equation}
Let $X^\Omega_t$ denote the killed process on $\Omega$, i.e.,
\begin{equation}\label{XDdef}
X_t^\Omega=\begin{cases}
		X_t, & \ t<\tau_\Omega,        \\
		\partial, & t\geq\tau_\Omega,
		\end{cases}
\end{equation}
where $\partial$ denotes a cemetery point.
We say that a boundary point $x$ of $\Omega$ is regular for $\Omega$ if $\px(\tau_\Omega=0)=1$. We say that $\Omega$ is regular if every boundary point of $\Omega$ is regular for $\Omega$.  We say that a Markov process $X_t$ on $\rd$ or its semigroup $P_tf(x)=\E_x[f(X_t)]$ is {\it strong Feller} if for any bounded measurable real-valued function $f$ with compact support on $\Rd$, $P_tf(x)$ is bounded and continuous on $\Rd$.  We say that a Feller process (resp, semigroup)  is {\it doubly Feller} if it also has the strong Feller property (e.g., see \cite{SW12}).

\begin{lemma}\label{lem:FellerProp}
\mmm{Suppose that $X_t$ is a doubly Feller process on $\Rd$ and that $\Omega$ is regular.} Then
\begin{equation}\label{PtDdef}
P_t^\Omega f(x):=\Ex [f(X_t^\Omega)],\quad x\in \Omega,\ t\geq 0
\end{equation}
defines a Feller semigroup on $C_0(\Omega)$.
\end{lemma}

\begin{proof}
Since $X_t$ is doubly Feller and $\Omega$ is regular, the theorem on page 68 of Chung \cite{Chung} implies that $X_t^\Omega$ is also doubly Feller. In particular, we have that $P_t^\Omega$ is a Feller semigroup on $C_0(\Omega)$.
\end{proof}

For any bounded domain $\Omega\subset\R^d$ we will say that $U$ is compactly contained in $\Omega$, and write $U\subset\subset \Omega$, if \TL{$\bar U$, the closure of $U$, defined as the intersection of all closed sets containing $U$, satisfies $\bar U\subset\Omega$.  Since $\Omega$ is bounded, $\bar U$ is compact for any $U\subset\subset\Omega$.}  If $f_n(x)\to f(x)$ for all $x\in \Omega$, and uniformly on $x\in U$ for any $U\subset\subset\Omega$, we say that $f_n\to f$ uniformly on compacta in $\Omega$.  If $P_t^\Omega $ is a Feller semigroup on $C_0(\Omega)$, then it has a generator
 \begin{equation}\label{GenDefOmega}
L_\Omega f:=\lim_{t\searrow0}\frac{P_t^\Omega f-f}{t} \quad\text{in $C_0(\Omega)$,}
\end{equation}
with domain $\D(L_\Omega)\subset C_0(\Omega)$.  The next result shows that this generator $L_\Omega$ can be computed using the pointwise formula \eqref{GenDefPtwise23} for the {\it original} Feller generator on $C_0(\Rd)$.  Given a function $f\in C_0(\Omega)$, we apply the formula \eqref{GenDefPtwise23} to the zero extension of $f$, i.e., we set $f(x)=0$ for all $x\notin\Omega$, to get an element of $C_0(\Rd)$.  Then we will write $L^\sharp f(x)\in C_0(\Omega)$ to mean that the function defined by \eqref{GenDefPtwise23} exists for $x\in\Omega$, is continuous on $\Omega$, and tends to zero as $x\in\Omega$ approaches the boundary.  This {\it does not} require the limit in \eqref{GenDefPtwise23} to exist for any $x\notin\Omega$.

\begin{theorem}\label{pwL}
Assume that $X_t$ is a doubly Feller process on $\R^d$, and let $\Omega\subset \R^d$ be a regular bounded domain.  Then the domain of the killed generator \eqref{GenDefOmega} is given by
\begin{equation}\begin{split}\label{domainsLE}
\D(L_\Omega)=&\{f\in C_0(\Omega):L^\sharp f\in C_0(\Omega)\} .
\end{split}\end{equation}
Also $L_\Omega f(x)=L^\sharp f(x)$ for all $x\in\Omega$, and \eqref{GenDefPtwise23} holds uniformly on compacta in $\Omega$.
 \end{theorem}

\begin{proof}
Since $X_t$ is a doubly Feller, it follows from Lemma \ref{lem:FellerProp} that $X_t^\Omega$ is a Feller process, whose semigroup \eqref{PtDdef} has a generator \eqref{GenDefOmega} on $C_0(\Omega)$.  Let $f\in \D(L_\Omega)$. Then there exists $g\in C_0(\Omega)$ such that
$$ g(x)=\lim_{t\to 0}\frac{P_t^\Omega f(x)-f(x)}{t}$$
for all $x\in \Omega$. Furthermore,
\begin{equation}\begin{split}
  P_t^\Omega f(x)-P_tf(x)=&\Ex f(X_t^\Omega)-\Ex f(X_t)\\
  =&\Ex[f(X^\Omega_t)I\{\tau_\Omega\ge t\}]+\Ex[f(X^\Omega_t)I\{\tau_\Omega<t\}]\\
  &-\Ex[f(X_t)I\{\tau_\Omega\ge t\}]-\Ex[f(X_t)I\{\tau_\Omega< t\}]\\
  =&f(X_{\tau_\Omega})-\Ex[f(X_t)I\{\tau_\Omega< t\}].
  \end{split}
\end{equation}
Indeed, as $X_t$ has a.s.\ right-continuous sample paths, the first and third term cancel. 
Hence \begin{equation}\begin{split}
  \frac{P_t^\Omega f(x)-f(x)}{t}-\frac{P_tf(x)-f(x)}{t}=& \frac{\Ex[(f(X_{\tau_\Omega})-f(X_t))I\{\tau_\Omega< t\}]}{t}.
  \end{split}
\end{equation}
By the Strong Markov Property \cite[Proposition 7.9]{KAL} we have
\begin{equation}\begin{split}\label{TLstar1}
\Ex[f(X_t)I\{\tau_\Omega< t\}]=&{\mathbb{E}}^x\left[{\mathbb{E}}^{X_{\tau_\Omega}}[f(X_{t-\tau_\Omega})]I\{\tau_\Omega< t\}\right].
\end{split}\end{equation}
Since $\Ex[f(X_t)]\to f(x)$ uniformly in $x\in\R^d$ as $t\to 0$, for any $\varepsilon>0$ there exists a $\delta>0$ such that
\begin{equation}\begin{split}\label{TLstar2}
\left|\Ex[(f(X_{\tau_\Omega})-f(X_t))I\{\tau_\Omega< t\}]\right|&=\left|{\mathbb{E}}^x\left[{\mathbb{E}}^{X_{\tau_\Omega}}[f(X_0)-f(X_{t-\tau_\Omega})]I\{\tau_\Omega< t\}\right]\right|\\
&\leq\varepsilon {\mathbb P}^x[\tau_\Omega< t]
\end{split}\end{equation}
for $0<t<\delta$. Given $U\subset\subset \Omega$, choose $r>0$ so that $B(x,r):=\{y\in\Rd:|x-y|< r\}\subset \Omega$ for all $x\in U$.  Let $\tau^x_r:=\inf\left\{t\geq 0:|X_t-\MMM{x}|\geq r\right\}$ for the process started at $\MMM{X_0=}\,x\in U$. Then
${\mathbb P}^x[\tau_\Omega<t]\leq{\mathbb P}^x[\tau^\MMM{x}_r<t]$, and
\MCUBED{by \cite[Theorem 5.1 and Proposition 2.27(d)]{BSW}} \TL{there exists some $M>0$ such that
\begin{equation}\label{TLstar3}
\frac{{\mathbb P}^x[\tau^\MMM{x}_r<t]}{t}<M,\quad\text{for all $x\in U$ and $t>0$}.
\end{equation}
Then
$$\frac{\left|P_t^\Omega f(x)-P_tf(x)\right|}{t}\leq \varepsilon\frac{{\mathbb P}^x[\tau^\MMM{x}_r<t]}{t}<\varepsilon M $$
for all $x\in U$ and $0<t<\delta$. Hence we have}
\begin{equation}\label{mmmUC}
\frac{P_t^\Omega f(x)-P_tf(x)}{t}\to 0 \quad\text{ uniformly on compacta in $\Omega$,}
\end{equation}
\TL{as $t\to0$}. Therefore any $f\in\TL{\D}(L_\Omega)$ is also contained in the set on the right-hand side of equation \eqref{domainsLE}, and in addition, \eqref{mmmUC} holds.

Conversely, suppose $f\in C_0(\Omega)$ and that $(P_tf(x)-f(x))/t\to g(x)$ as $t\to 0$ for some $g\in C_0(\Omega)$, for all $x\in \Omega$. As $L_\Omega$ is the generator of a contraction semigroup on $C_0(\Omega)$, its resolvent $(\lambda I-L_\Omega)^{-1}$ exists for all \TL{$\lambda>0$}, and maps $C_0(\Omega)$ onto $\D(L_\Omega)$ \TL{\cite[Chapter VII, Proposition (1.4)]{RY}}. Then there exists some $h\in \D(L_\Omega)$ such that $(I-L_\Omega)h=f-g$. \TL{By \eqref{mmmUC} applied to $h$,} $$L_\Omega h(x)-g(x)=\lim_{t\to 0}\frac{P_th(x)-h(x)-(P_tf(x)-f(x))}{t},\quad \TL{x\in\Omega}.$$ Hence, for $u=h-f$ we get
$$u(x)=\lim_{t\to 0}\frac {P_tu(x)-u(x)}{t},\quad \TL{x\in\Omega}.$$
Without loss of generality let $x_0\in\Omega$ be such that $\|u\|=\sup_{x\in \Omega}|u(x)|=u(x_0)>0$ (otherwise consider $-u$). Since $P_t$ is a contraction, $P_tu(x_0)\leq\|P_t u\|\leq\|u\|= u(x_0)$ and therefore
$$0\ge \left(P_tu(x_0)- u(x_0)\right)/t\to u(x_0)>0$$ \TL{as $t\to 0$}, which is a contradiction.
Hence $\sup_{x\in \Omega}|u(x)|=0$ and therefore $h=f$.  Thus any $f$ in the set on the right-hand side of equation \eqref{domainsLE} is also an element of $\D(L_\Omega)$.
\end{proof}

\begin{remark}
Here we sketch an alternate proof \MMM{that $L_\Omega f(x)=L^\sharp f(x)$ for all $x\in\Omega$ and }\BB{ $f\in\D(L)\cap\D(L_\Omega)$} [thanks to Zhen-Qing Chen].  Define
\[M_t^f:=f(X_t)-f(X_0)-\int_0^t Lf(X_s)ds\]
and note that, by Kallenberg \cite[Lemma 17.21]{KAL}, $M_t^f$ is a martingale for any $f\in \D(L)$, and
\[\Ex\left[M_{t\wedge\tau_\Omega}^f\right]:=\Ex\left[f(X^\Omega_t)-f(X^\Omega_0)-\int_0^t Lf(X^\Omega_s)ds\right]=0\]
for any $t>0$. Hence we have (pointwise) for any $f\in \D(L)$ and any $x\in D$ that
\begin{equation*}\begin{split}
\lim_{t\to 0+} \frac{P^\Omega_t f(x) -f(x)}{t}&=\lim_{t\to 0+} \frac{\Ex\left[f(X^\Omega_t)\right] -f(x)}{t}\\
&=\lim_{t\to 0+} t^{-1}\Ex\left[f(X^\Omega_t)-f(X^\Omega_0)\right]\\
&=\lim_{t\to 0+} t^{-1}\Ex\left[\int_0^t Lf(X^\Omega_s)\,ds\right] \\
&=\lim_{t\to 0+} t^{-1}\int_0^t \Ex\left[Lf(X^\Omega_s)\right]\,ds =Lf(x)
\end{split}\end{equation*}
assuming that $Lf(x)$ is continuous.  Here we use Lebesgue's
differentiation theorem (e.g., see \cite[Theorem 7.16]{Wheeden}) and the fact that $X^\Omega_t$ is continuous in probability.  \MCUBED{Note, however, that $\D(L_\Omega)$ typically contains functions that cannot be extended to an element of $\D(L)$, see for example Remark \ref{HarishRemark}.}
\end{remark}

Next we show that functions in $\D(L_\Omega)$ can be characterized as functions in $C_0(\Omega)$ that are locally in the domain of $L$. This will be used for \fe{explicitly} computing the killed generator.

\begin{theorem}\label{master}
Assume that $X_t$ is a doubly Feller process on $\R^d$, and let $\Omega\subset \R^d$ be a regular bounded domain.  Then
  \begin{equation}\begin{split}\label{DL_E}
    \D(L_\Omega )=
    \{f\in C_0(\Omega ):&\ \exists g\in C_0(\Omega ),(f_n)\subset \D(L) \mbox{ we have } f_n\to f \mbox{ in }C_0(\R^d) \\&
     \mbox{ and } Lf_n\to g \mbox{ unif. on compacta in } \Omega  \} ,
    \end{split}
  \end{equation}
and for $f,g$ as in \eqref{DL_E} we have $L_\Omega f=g.$
\end{theorem}

\begin{proof}
First we show that the limit $g$ in \eqref{DL_E} is unique for any given $f$. Assume that for some $f_n\in \D(L)$ we have $f_n\to 0$ uniformly on $\R^d$ and $Lf_n(x)\to g(x)$ for all $x\in \Omega $, \TL{uniformly on compacta in $\Omega$}. We claim that $g(x)=0$ for all $x\in \Omega $. Assume $g(x)> \delta$ for all $x\in B(x_0;r)\subset \Omega $ for some $x_0\in \Omega $ and $\delta,r>0$.  Choose $h\in C_c^\infty$ such that $h(x_0)>0$ is the only local maximum. Let $\epsilon>0 $ be small enough that $\TL{U}=\{x:h(x_0)-h(x)<\epsilon\}\subset B(x_0,r)$ and let $y=\sup_{x\in \Omega}|Lh(x)|$. Consider
$$h_n=h+4\frac{y}{\delta}f_n.$$ Let $n$ be large enough such that $|4\frac{y}{\delta} f_n(x)|<\epsilon/2$ \MCUBED{for all $x\in \Omega$}
and $Lf_n(x)>\delta/2$ for all $x\in \TL{U}$.
\MCUBED{Then
\[4\frac{y}{\delta} Lf_n(x)>4\frac{y}{\delta} \frac\delta{2}=2y\quad\text{for all $x\in \Omega$,}\]
and since $Lh(x)\leq y$ for all $x\in \Omega$, it follows that $Lh_n(x)>y$ for all $x\in \Omega$.  For all $x\notin U$ we have $h(x)\leq h(x_0)-\varepsilon$, and hence $h_n(x)\leq h(x_0)-\varepsilon/2$ for all $x\notin U$.  Since $h_n(x_0)>h(x_0)-\varepsilon/2$, it follows that $h_n$ attains its maximum at some point $x_n\in U$.} \MCUBED{Then the positive maximum principle \cite[Theorem 17.11 (iii)]{KAL} implies that $Lh_n(x_n)\leq 0$, and this contradicts the fact that  $Lh_n(x)>0$ for all $x\in \Omega$.}  Hence $g\le 0$. Considering the sequence $-f_n$, we obtain that $-g\le0$ and hence $g=0$.  Given two sequences $f_n$ and $f_n'$ in $\D(L)$ that both converge to $f$ in $C_0(\R^d)$, and such that $Lf_n\to g$ and $Lf_n'\to g'$ in $C_0(\R^d)$, it follows that $f_n-f_n'\to 0$ in $C_0(\R^d)$, and hence $L(f_n-f_n')\to g-g'=0$, which proves uniqueness.

Next we show that functions $f\in \D(L_\Omega )$ can be approximated locally in the graph norm by functions in the domain of $L$, namely by the functions
$$f_\lambda= (\lambda-L)^{-1}\lambda f.$$ As $P_t f$ is continuous in $t$ and $\|P_tf\|\leq\|f\|$, it is not hard to check that $f_\lambda=\lambda \int_0^\infty e^{-\lambda t} P_tf\,dt$ and
$$\lim_{\lambda\to\infty}f_\lambda=P_0f=f$$ \TL{in $C_0(\R^d)$}. Furthermore,
$f_\lambda\in \D(L)$ and by definition,
$$Lf_\lambda=\lambda f_\lambda-\lambda f.$$
Theorem \ref{pwL} implies that $\frac{P_tf(x)-f(x)}{t}\to L_\Omega f(x)$ uniformly in $x\in U\subset\subset \Omega $, and then it is not hard to check that, \MMM{using a substitution $u=\lambda t$},
\begin{equation}\begin{split}\lim_{\lambda\to\infty}Lf_\lambda(x)=&\lim_{\lambda\to\infty}\lambda^2\int_0^\infty e^{-\lambda t}P_tf(x)\,dt-\lambda f(x)\\
=&\lim_{\lambda\to\infty}\lambda^2\int_0^\infty e^{-\lambda t}(P_tf(x)-f(x))\,dt\\
=&\lim_{\lambda\to\infty}\lambda^2\int_0^\infty te^{-\lambda t}\frac{P_tf(x)-f(x)}{t}\,dt\\
=&\MMM{\lim_{\lambda\to\infty}\int_0^\infty ue^{-u}\frac{P_{(u/\lambda)}f(x)-f(x)}{(u/\lambda)}\,du}\\
=&L_\Omega f(x)
 \end{split}\end{equation}
uniformly in $x\in U$. Hence $\D(L_\Omega )$ is contained in the set on the right-hand side of \eqref{DL_E}.

To prove the reverse set inclusion, suppose that \TL{$f\in C_0(\Omega )$ and} for some $f_n\in \D(L)$ we have $f_n\to f$ in \TL{$C_0(\R^d)$} and $Lf_n(x)\to g(x)$ uniformly in $x\in U\subset\subset \Omega $ for some $g\in C_0(\Omega )$. Let
$h=(I-L_\Omega )^{-1}(f-g)$ so that $$h-f=L_\Omega h-g.$$
Since the resolvent maps $C_0(\Omega )$ onto $\D(L_\Omega )$, the function $h$ lies in the set on the right-hand side of \eqref{DL_E} by what we have already proven. Hence there exist $h_n\in \D(L)$ such that $h_n\to h$ in $C_0(\R^d)$ and $Lh_n(x)\to L_\Omega  h(x)$ for all $x\in \Omega $, uniformly on compacta. Let $u=h-f$ and assume (without loss of generality) that $u(x_0)=\|u\|>\epsilon$ for some $\epsilon>0$. Let $u_n=h_n-f_n$ so that $Lu_n(x)\to L_\Omega h(x)-g(x)=u(x)$ uniformly in $x\in U\subset\subset\Omega$. However, as $u_n$ converges uniformly to $u$ there exists $N>0$ and $U\subset\subset \Omega $ such that $\{x_n:u_n(x_n)=\TL{\|u_n\|}\}\subset U$ for all $n>N$. As $u_n(x_n)>\epsilon/2$ for large $n$ and $Lu_n(x_n)\le 0$ by the maximum principle \MCUBED{\cite[Theorem 17.11 (iii)]{KAL}},
$u_n(x)-Lu_n(x)$ cannot converge uniformly on $U$ to 0. This is a contradiction, and hence $u\equiv 0$.  Then $h=f\in \D(L_\Omega )$, which completes the proof.
\end{proof}

\FE{Even if $f\not\in C_0^2(\R^d)$, the pointwise limit \eqref{GenDefPtwise23} might still exist for some $x\in\R^d$}. \fe{The next result shows} that \FE{we still have $L^\sharp f(x)=L^p f(x)$} for functions that are locally twice differentiable.

\begin{lemma}\label{ptwise}Assume that $X_t$ is a Feller process on $\R^d$ with $C_c^\infty(\R^d)$ contained in the domain of its generator.
  Let $f\in C_0(\R^d)$ with $f$ twice continuously differentiable in a neighborhood $\Omega$ of $x$. Then
  \FE{$L^\sharp f(x)=L^p f(x)$ where $L^p$} is given by \eqref{PDO}.
\end{lemma}

\begin{proof}
Let $r$ be such that $B(x,2r)\subset \Omega$ and pick $f_n\in C_0^2(\R^d)$ with the property that $f_n\to f$ uniformly and $f_n(y)=f(y)$ for all $y\in B(x,r)$.  Then
  \begin{equation}
    \begin{split}
      \left|L^\sharp f(x)-\FE{L^p f}(x)\right|=&\left|L^\sharp f(x)-Lf_n(x)+Lf_n(x)-\FE{L^p f}(x)\right|\\
      =&\left|\lim_{t\searrow 0}\frac{\Ex[f(X_t)-f_n(X_t)]}{t}\right.\\&+\left.\int_{\R^d\setminus\{0\}} \left(f(x+y)-f_n(x+y)\right)\,N(x,dy)\right|\\
      \le&\lim_{t\searrow 0}\frac{\mathbb P^x\{\tau_r^x<t\}}{t}\|f-f_n\|\\
      &+\left|\int_{|y|>r} \left(f(x+y)-f_n(x+y)\right)\,N(x,dy)\right|\\
      \le& M_r\|f-f_n\|\to 0,
    \end{split}
  \end{equation}where $M_r=C_r+N_r$ with $C_r$ as in \cite[Theorem 5.1]{BSW} given by  $$\mathbb P^x\{\tau_r^x<t\}\le t C_r$$ and $N_r=C/r^2$ with $C$ given as in  \cite[Theorem 2.31b]{BSW} by $$\int_{\R^d\setminus\{0\}}\min(|y|^2,1)N(x,dy)<C.$$
  This concludes the proof.
\end{proof}

The following theorem is the main result of this section.  It shows that we can evaluate the generator $L_\Omega f(x)$ of the killed Markov process pointwise for $x\in \Omega $ using the explicit formula \eqref{PDO} for $Lf(x)$.  Let $C^2_0(\Omega )$ denote the set of $C_0(\Omega )$ functions with first and second order partial derivatives that are continuous at every  $x\in \Omega $.  Observe that $f\in C_0^2(\Omega )$ {\it does not} require the partial derivatives to remain bounded as $x\in \Omega $ approaches the boundary of the domain.

\begin{theorem}\label{masterCor}
Assume that $X_t$ is a doubly Feller process on $\R^d$, and let $\Omega\subset \R^d$ be a regular bounded domain.  \FE{Suppose that $C_0^2(\R^d)$ is a core for $L$, so that $Lf(x)=L^\sharp f(x)=L^p f(x)$ for every $x\in\R^d$ and $f\in  C^2_0(\R^d)$}. Then:
 \begin{enumerate}
 \item for every $f\in \D(L_\Omega)$ there exists $f_n\in C^2_0(\Omega)$ such that $f_n\to f$ uniformly and $\FE{L^p} f_n$ converges uniformly on compact subsets of $\Omega$ to $L_\Omega f$;
 \item if $f_n\in C^2_0(\Omega)$ is such that $f_n\to f\in C_0(\Omega)$ uniformly and $\FE{L^p}f_n\to g\in C_0(\Omega)$ converges uniformly on compact subsets of $\Omega$, then $f\in \D(L_\Omega)$ and $L_\Omega f=g$.
 \end{enumerate}
In particular,
if $f\in C^2_0(\Omega )$ and $\FE{L^p} f(x)\in C_0(\Omega )$,
then $f\in \D(L_\Omega )$ and $L_\Omega f(x)=\FE{L^p} f(x) $ is given by \eqref{PDO} for every $x\in\Omega$.
\end{theorem}

\begin{proof} Consider a sequence of open sets $\Omega _n\subset\subset \Omega _{n+1}$ for $n\geq 1$ with $\bigcup \Omega _n=\Omega $. Take $\psi_n\in C^\infty_c(\rd)$ with $I_{\Omega _{n}}\le\psi_n\le I_{\Omega _{n+1}}$.

To prove (1), by Theorem \ref{master} and the definition of a core there exists $f_n^\infty\in C_0^2(\R^d)$ such that $f_n^\infty\to f$ uniformly and $\FE{Lf_n^\infty(x)=L^p} f_n^\infty(x)\to L_\Omega f(x)$ uniformly on compact subsets of $\Omega$. \fe{To see this, suppose that $f\in \D(L_\Omega)$, and extend $f$ to an element of $C_0(\R^d)$ by setting $f(x)=0$ for $x\notin\Omega$.  Apply Theorem \ref{master} to obtain a sequence $(f_n)\subset \D(L)$ such that $f_n\to f$ in $C_0(\R^d)$, and $Lf_n(x)\to L_\Omega f(x)$ uniformly on compacta in $\Omega$.  Then for any compact set $U\subset\Omega$ and any integer $k>0$, for some $n_0$, we have $\|f_n-f\|<1/k$ for all $n\geq n_0$, and $|Lf_n(x)-L_\Omega f(x)|<1/k$ for all $x\in U$ and all $n\geq n_0$.  Since $C_0^2(\R^d)$ is a core, for each $f_n$ there exists a sequence $f_{nm}^\infty\in C_0^2(\R^d)$ such that $\|f_n-f_{nm}^\infty\|+\|Lf_n-Lf_{nm}^\infty\|\to 0$ as $m\to\infty$.  Hence for any $n>0$ there is an $m_0$ such that $\|f_n-f_{nm}^\infty\|+\|Lf_n-Lf_{nm}^\infty\|<1/n$ for all $m\geq m_0$.  Define $f_{n}^\infty=f_{nm_0}^\infty$.  Then for $n\geq n_1:=\max(n_0,k)$ we have for all $n\geq n_1$ that, by the triangle inequality, $\|f-f_{n}^\infty\|<2/k$ for all $n\geq n_1$ and $|L_\Omega f(x)-\FE{L^p} f_{n}^\infty(x)|<2/k$ for all $x\in U$ and all $n\geq n_1$.   }  Let $f_n=\psi_nf_n^\infty$. Then $f_n\in C_0^2(\Omega)$ and by Lemma \ref{ptwise}
$$\left|\FE{L^p} f_n(x)-\FE{L^p} f_n^\infty(x)\right|=\left|\int_{x+y\not\in\Omega_n} \left(f_n(x+y)-f_n^\infty(x+y)\right)\,N(x,dy)\right|.$$
Let $U$ be a compact subset of $\Omega$. Then there exists $n_0$ such that $U\subset \Omega_{n_0}$ and since the closure of $\Omega_{n_0}$ is compact, \fe{it is not hard to check that for any $n>n_0$, there exists some $\epsilon>0$} such that $z\not\in\Omega_n$ implies that $|z-x|>\epsilon$ for all $x\in\fe{\Omega_{n_0}}$.

\fe{To see this, write $B(x,r)=\{w:|w-x|<r\}$ and note that, since $\bar\Omega_{n_0}\subset\Omega_n$ open, for each $x\in \bar\Omega_{n_0}$ there exists some $r>0$ such that $B(x,2r)\subset\Omega_n$.  The collection of sets $\{B(x,r):x\in\bar\Omega_{n_0}\}$ covers the compact set $\bar\Omega_{n_0}$, hence there exists a finite subcover $B(x_j,r_j)$ for $j=1,\ldots,J$ such that $\bar\Omega_{n_0}\subset\bigcup_{j=1}^J B(x_j,r_j)$. For any $x\in\bar\Omega_{n_0}$ we have $|x-x_j|<r_j$ for some $j=1,\ldots,J$ and $|x_j-z|>2r_j$ for all $z\notin\Omega_{n}$, so that $|x-z|\geq |x_j-z|-|x-x_j|>r_j$.  Then the claim holds with $\varepsilon=\min\{r_j:1\leq j\leq J\}$.}

By \cite[\fe{Proposition 2.27 (d)}]{BSW},
$$\int_{x+y\not\in\Omega_n} \left|f_n(x+y)-f_n^\infty(x+y)\right|\,N(x,dy)\le \|f_n-f_n^\infty\|\int_{|y|>\epsilon}\,N(x,dy)\to 0$$
 uniformly on $U$.  \fe{} and hence $\FE{L^p} f_n$ converges uniformly on $U$ to $L_\Omega f$.

To prove (2), let $f_n^\infty=\psi_n f_n$. Then $f_n^\infty\in C_0^2(\R^d)$ and, with the same argument as above, $\FE{L^p} f_n^\infty\to g$ uniformly on compact subsets of $\Omega$. By Theorem \ref{master}, $f\in \D(L_\Omega)$ and $L_\Omega f=g$.
\end{proof}

\begin{remark}\label{extension}
\fe{In general, we do not know whether $L_\Omega f(x)$ can be computed by the pointwise formula \eqref{PDO} for every $f\in\D(L_\Omega)$.  However, Theorem \ref{masterCor} shows that we can always write $L_\Omega f(x)=\lim_{n\to\infty} \FE{L^p} f_n(x)$ for some $f_n\in C^2_0(\Omega)$, so that the pointwise formula \eqref{PDO} applies to $\FE{L^p} f_n(x)$.  Hence $L_\Omega$ is the unique continuous extension to $\D(L_\Omega)$ of the formula \eqref{PDO} on $C^2_0(\Omega)$, \FE{compare \cite[Theorem 2.37 (a)]{BSW}}.  This is similar to the manner in which the Fourier transform is defined as an isometry on $L_2(\R^d)$:  The pointwise definition is valid on a dense subset $L_1(\R^d)\cap L_2(\R^d)$, and the isometry is the unique continuous extension to $L_2(\R^d)$.}
\end{remark}

\begin{remark}\label{RemHawkes}
In the case $c\equiv 0$, $l(x)\equiv l$,  $Q(x)\equiv Q$, and $N(x,dy)\equiv N(dy)$, \eqref{PDO} is the generator of a L\'evy process on $\R^d$.  Then Hawkes \cite[Lemma 2.1]{Hawkes} shows that $X_t$ is doubly Feller if and only if $X_t$ has a Lebesgue density for each $t>0$. It follows from Sato \cite[Theorem 31.5]{SAT} that $C_c^\infty(\R^d)\subset C^2_0(\R^d)\subset \D(L)$ in this case.  Hence the conditions of Theorem \ref{masterCor} are satisfied for any L\'evy process with a density.
\end{remark}

\section{Fractional Cauchy problems}\label{sec3}
In this section, we recall some results on (fractional) Cauchy problems that will be useful in Section \ref{SecApp}.
If $X_t$ is a doubly Feller process on $\Rd$ and $\Omega$ is a regular bounded domain, then Lemma \ref{lem:FellerProp} implies that the semigroup $P_t^\Omega$ associated with the killed process, defined by \eqref{PtDdef},  is a Feller semigroup on $C_0(\Omega)$.  The generator $L_\Omega$ of this semigroup and its domain $\D(L_\Omega)$ are given in Theorem \ref{pwL}. \fe{If $C_c^\infty(\R^d)\subset \D(L)$, then Theorem \ref{masterCor} gives an explicit pointwise formula \eqref{PDO} for $L_\Omega$, valid for all $x\in\Omega$ and all $f\in  C^2_0(\Omega)$.  Remark \ref{extension} explains that $L_\Omega$ is the unique continuous extension of \eqref{PDO} to $\D(L_\Omega)$.} Since $P_t^\Omega$ is a Feller semigroup, the function $u(t)=P_t^\Omega f$ solves the abstract Cauchy problem
 \begin{equation}\begin{split}\label{CpGenOmega}
\partial_t u(x,t)&= L_\Omega u(x,t)\quad u(x,0)=f(x)
\end{split}\end{equation}
for any $f\in\D(L_\Omega)$, e.g., see \cite[Proposition 3.1.9 (h)]{ABHN2}.  Furthermore $P_t^\Omega f$ is a {\it mild solution} to the Cauchy problem \eqref{CpGenOmega} for {\it any} $f\in C_0(\Omega)$ \cite[Proposition 3.1.9 (b)]{ABHN2}.  That is, $u(x,t)=P_t^\Omega f(x)$ is the unique solution in $C_0(\Omega)$ to the corresponding integral equation
 \begin{equation}\begin{split}\label{CpGenOmegaMild}
u(t)&= f+L_\Omega \int_0^t u(s)\,ds
\end{split}\end{equation}
for all $t\geq 0$.

The function
 \begin{equation}\label{CpGenOmegaInhom}
 u(x,t)=P_{t}^\Omega f(x)+\int_0^t P_{s}^\Omega g(x,t-s)\,ds
 \end{equation}
is the unique solution to the inhomogeneous Cauchy problem
\begin{equation}\begin{split}\label{CPinhom}
\partial_t u(x,t)&=L_\Omega u(x,t) +g(x,t);\quad u(x,0)=f(x)
\end{split}\end{equation}
for any $g(x,t)=g_0(x)+\int_0^t \partial_s g(x,s)\,ds\in C_0(\Omega)$ such that $\partial_t g(x,t)\in L^1_{\rm loc}(\RR^{+},C_0(\Omega))$ \cite[Corollary 3.1.17]{ABHN2}.  The same formula \eqref{CpGenOmegaInhom} gives the unique mild solution to \eqref{CPinhom} for any $f\in C_0(\Omega)$ and any $g\in L^1([0,T),C_0(\Omega))$, see \cite[Theorem 3.1.16]{ABHN2}.  That is, it solves the integral equation
 \begin{equation}\begin{split}\label{CpGenOmegaMildInhom}
u(t)&= f+L_\Omega \int_0^t u(s)\,ds+\int_0^t g(s)\,ds .
\end{split}\end{equation}
In practice, the condition $f\in\D(L_\Omega)$ can be hard to check.  In numerical analysis theory, it is therefore common to prove results like the Lax Equivalence Theorem for mild solutions, which can then be approximated by strong solutions, see for example \cite[Chapter 10]{IK}.

The positive and negative Riemann-Liouville fractional integrals of a suitable function $f:\R \to \R$ are defined by
\begin{equation}\begin{split}\label{RLint}
{\mathbb I}_{[L,x]}^\alpha f(x)&=\frac {1}{\Gamma(\alpha)} \int_L^x f(y) (x-y)^{\alpha-1}\,dy\\
{\mathbb I}_{[x,R]}^\alpha f(x)&=\frac {1}{\Gamma(\alpha)} \int_x^R f(y) (y-x)^{\alpha-1}\,dy
\end{split}\end{equation}
for any $\alpha>0$ and any $-\infty\leq L<x<R\leq\infty$, see for example \cite[Definition 2.1]{Samko}.
The positive and negative Riemann-Liouville fractional derivatives are defined by
\begin{equation}\begin{split}\label{RLdef}
{\mathbb D}^\alpha_{[L,x]} f(x)&=\left(\frac{d}{dx}\right)^n {\mathbb I}_{[L,x]}^{n-\alpha} f(x)=\frac{1}{\Gamma(n-\alpha)}\frac{d^n}{dx^n}\int_L^x f(y)(x-y)^{n-\alpha-1}dy\\
{\mathbb D}^\alpha_{[x,R]} f(x)&=\left(-\frac{d}{dx}\right)^n {\mathbb I}_{[x,R]}^{n-\alpha} f(x)=\frac{(-1)^n}{\Gamma(n-\alpha)}\frac{d^n}{dx^n}\int_x^R f(y)(y-x)^{n-\alpha-1}dy
\end{split}\end{equation}
for any non-integer $\alpha>0$ and any $-\infty\leq L<x<R\leq\infty$, where $n-1<\alpha<n$, see for example \cite[p.\ 31]{Samko}.  The positive and negative Caputo fractional derivatives are defined by
\begin{equation}\begin{split}\label{CaputoDef}
{\partial}^\alpha_{[L,x]} f(x)&={\mathbb I}_{[L,x]}^{n-\alpha} \left(\frac{d}{dx}\right)^n f(x)=\frac{1}{\Gamma(n-\alpha)}\int_L^x f^{(n)}(y)(x-y)^{n-\alpha-1}dy\\
{\partial}^\alpha_{[x,R]} f(x)&={\mathbb I}_{[x,R]}^{n-\alpha} \left(-\frac{d}{dx}\right)^n f(x)=\frac{(-1)^n}{\Gamma(n-\alpha)}\int_x^R f^{(n)}(y)(y-x)^{n-\alpha-1}dy ,
\end{split}\end{equation}
see for example \cite[Eq.\ (2.16)]{FCbook}.  If $0<\beta<1$, then for a function $f:\RR^{+}\to\RR$ with Laplace transform
\begin{equation}\label{LTdef}
\tilde f(s):=\int_0^\infty e^{-st}f(t)\,dt
\end{equation}
it is not hard to show that ${\partial}^\beta_{[0,t]} f(t)$ has Laplace transform $s^\beta \tilde f(s)-s^{\beta-1}f(0)$, extending the well-known formula for integer order derivatives.  Since ${\mathbb D}^\beta_{[0,t]} f(t)$  has Laplace transform $s^\beta \tilde f(s)$, and since $s^{\beta-1}$ is the Laplace transform of the function $t^{-\beta}/\Gamma(1-\beta)$, it follows by the uniqueness of the Laplace transform that
\begin{equation}\label{RLtoCaputo}
{\partial}^\beta_{[0,t]} f(t)={\mathbb D}^\beta_{[0,t]} f(t)-\frac{t^{-\beta}}{\Gamma(1-\beta)}f(0) ,
\end{equation}
see \cite[p.\ 39]{FCbook} for more details.

Let $g_\beta(u)$ denote the probability density function of the standard stable subordinator, with Laplace transform
\begin{equation}\label{subordLT}
\int_0^\infty e^{-su}g_\beta(u)du=e^{-s^\beta}
\end{equation}
for some $0<\beta<1$. Suppose that $D_t$ is a L\'evy process such that $g_\beta(u)$ is the probability density of $D_1$, and define the {\em inverse stable subordinator} (first passage time)
\begin{equation}\label{EtDef}
E_t=\inf\{u>0: D_u>t\} .
\end{equation}

A general result from the theory of semigroups \cite[Theorem 3.1]{fracCauchy} (see also Remark \ref{uniquenessRemark}) implies that the function
\begin{equation}\label{e:3.8}
v(x,t):=\int_0^\infty  g_\beta(r) P_{(t/r)^\beta}^\Omega f(x)\,dr
\end{equation}
\mmm{is the unique solution} to the time-fractional Cauchy problem
\begin{equation}\label{fCP}
{\mathbb D}^\beta_t v(x,t)=L_\Omega v(x,t)+\frac{t^{-\beta}}{\Gamma(1-\beta)}f(0);\quad v(x,0)=f(x)
\end{equation}
for any $f\in\D(L_\Omega)$.
Using \eqref{RLtoCaputo}, it follows that the same function also solves
\begin{equation}\begin{split}\label{CaputofCP}
\partial^\beta_t v &=L_\Omega v ;\quad v(0)=f
\end{split}\end{equation}
for any $f\in\D(L_\Omega)$.  Since
\begin{equation}\label{EtDens}
h(w,t)=\frac{t}{\beta}w^{-1-1/\beta}g_\beta(tw^{-1/\beta})
\end{equation}
is the probability density function of the inverse stable subordinator $E_t$ \cite[Corollary 3.1]{limitCTRW}, it follows by a simple change of variables that
\begin{equation}\label{FCPsoln2}
v(x,t)=\int_0^\infty h(w,t) P_{w}^\Omega f(x)\,dw =\int_0^\infty u(x,w)h(w,t)\,dw =\Ex [f(X^\Omega_{E_t})] .
\end{equation}

\begin{remark}\label{uniquenessRemark}
\mmm{The proof in \cite[Theorem 3.1]{fracCauchy} uses Laplace transforms, and although it is not explicitly stated, this also leads to a simple proof of uniqueness:  If $v(x,t)$ solves the fractional Cauchy problem \eqref{fCP}, then its Laplace transform
satisfies $\tilde v= (s^\beta-L)^{-1} s^{\beta-1} f$.   As $L$ generates a semigroup, $(s^\beta-L)^{-1}$ is a bounded operator for all $s^\beta$ in the right half plane. In particular $(s^\beta-L)^{-1} 0=0 $ and hence by the uniqueness of the Laplace transform, we have $v=0$ for initial data $f=0$. Then, given two solutions $v_1,v_2$ to \eqref{fCP}, their difference $v=v_1-v_2$ solves \eqref{fCP} with $f=0$, and hence $v_1=v_2$.  Therefore, \eqref{e:3.8} is the unique solution to the fractional Cauchy problem \eqref{fCP}.  The uniqueness of solutions is well know, and was used, for example, in \cite{IBM}.}
\end{remark}

Baeumer et al.\ \cite{forcing} consider the inhomogeneous fractional Cauchy problem
\begin{equation} \label{FCPforce1c}
\partial_t^\beta v(x,t)=L_\Omega v(x,t)+r(x,t);\;\;\; v(x,0)=f(x)
\end{equation}
with $0<\beta<1$.  Assuming that $t\mapsto v(x,t)$ is differentiable and $r(x,0)\equiv 0$, they show that \eqref{FCPforce1c} can also be written in Volterra integral form
\begin{equation} \label{FCPforceI}
v(x,t)=L_\Omega \,{\mathbb I}_{[0,t]}^\beta v(x,t)+f(x)+\int_0^t R(x,s)\,ds
\end{equation}
with $R(x,t)=\partial_t^{1-\beta} r(x,t)$ (and then
$R(x,t)={\mathbb D}_t^{1-\beta} r(x,t)$ as well).  Note that the forcing function $R(x,t)$ has the traditional meaning, and the units of $x/t$, unlike the function $r(x,t)$.  Any solution to the integral equation \eqref{FCPforceI} will be called a {\em mild solution} to the inhomogeneous fractional Cauchy problem \eqref{FCPforce1c}.  Then the inhomogeneous fractional Cauchy problem
\eqref{FCPforce1c}
with $r(x,0)\equiv 0$, and $R(t)\in L^1_{\rm loc}(\RR^{+};C_0(\Omega))$ has a unique mild solution
\begin{equation}\label{e:solD}
v(x,t)=\int_0^\infty P_s^\Omega f(x)h(s,t)\,ds+\int_0^t \int_0^\infty P_u^\Omega R(x,s) h(u,t-s)\,du\,ds,
\end{equation}
where $h$ is given by \eqref{EtDens}, see Baeumer et al.\ \cite[Theorem 1]{forcing}.

\section{Applications}\label{SecApp}
In many applications, including numerical analysis, it is necessary to consider fractional partial differential equations on a bounded domain with Dirichlet boundary conditions.  However, the theoretical foundations have been lacking.  Using the results of Section \ref{sec2} on the generator of the killed process, along with the results from Section \ref{sec3} on fractional Cauchy problems, we can establish existence and uniqueness of solutions to many fractional partial differential equations on a bounded domain with Dirichlet boundary conditions.  The main technical condition is that the underlying Markov process is doubly Feller (defined just before Lemma \ref{lem:FellerProp}).  In this section, we provide some example applications to illustrate the power of our method.

\begin{example}
\fe{This example clarifies that the solution to the Cauchy problem \eqref{CpGenOmega} on the bounded domain need not solve the corresponding Cauchy problem $\partial_t u(x,t)=L u(x,t)$ on $C_0(\R^d)$.}  Suppose that $f\geq 0$ is a smooth function with compact support in $\Omega=(0,M)\subset{\mathbb R}^1$, and that $L=\Delta=\partial_x^2$, the generator of a Brownian motion $X_t$ on ${\mathbb R}^1$.  The Cauchy problem
\begin{equation}\label{CPremarkEq}
\partial_t U(x,t)=\Delta U(x,t);\quad U(x,0)=f(x)
\end{equation}
has a unique solution
\[U(x,t)=\int_{y\in \Rd} f(y)p(x-y,t)dy\]
on $C_0({\mathbb R}^1)$, where
$p(x,t)=(4\pi t)^{-1/2}e^{-y^2/(4t)}$
is the Gaussian density with mean zero and variance $2t$.  Then $U(x,t)>0$ for all $t>0$ and all $x\in\Rd$, so $U(x,t)$ does not vanish off $\Omega$, and hence is not a solution to \eqref{CpGenOmega}.  In this case, the solution to \eqref{CpGenOmega} can be written explicitly in the form
\[u(x,t)=\sum_{n=1}^\infty f_n e^{-\lambda_n t}\psi_n(x)\]
where $\lambda_n=(n\pi/M)^2, n=1,2,3,\ldots$ are the eigenvalues and $\psi_n(x)=\sin(n\pi x/M)$ are the corresponding eigenfunctions of the generator $L_\Omega$ of the killed semigroup, and $f_n=(2/M)\int \psi_n(x)f(x)\,dx$, see for example \cite[Eq.\ (8) with $\alpha=1$]{agrawal}.  \fe{This solution belongs to $C^2_0(\Omega)$ for each $t\geq 0$, and hence we have $L_\Omega u(x,t)=\Delta u(x,t)$ for all $x\in\Omega$ and all $t>0$.  Hence the function $u(x,t)$ also} solves the differential equation $\partial_t u(x,t)=\Delta u(x,t)$, with the same initial condition $u(x,0)=f(x)$, \fe{at every point $(x,t)\in\Omega\times(0,\infty)$}. {One can extend the solution $u(x,t)$ to an element of $C_0({\mathbb R}^1)$ by setting $u(x,t)=0$ for $x\notin \Omega$.  However, this function cannot be \fe{twice differentiable in $x$} at the boundary points $x=0,M$, otherwise $u(x,t)$ would be another solution to \eqref{CPremarkEq} on $C_0({\mathbb R}^1)$, which would violate uniqueness.  }
\end{example}

\begin{example}\label{Ex1new}
Here we compute the generator of a killed stable process $X_t$ on $\R$ with index $1<\alpha<2$ in terms of fractional derivatives, see Theorem \ref{th:stable1d}.
Given a suitable function $f:\R\to\R$, the {\it generator form} of the positive fractional derivative is defined by
\begin{equation}\label{RL2p}
{\mathbf D}^\alpha_{(-\infty,x]} f(x):=\frac{\alpha(\alpha-1)}{\Gamma(2-\alpha)}\int_0^\infty \left[f(x-y)-f(x)+yf'(x)\right]y^{-1-\alpha}dy
\end{equation}
for $1<\alpha<2$ \cite[Eq.\ (2.18)]{FCbook}.
The {generator form} of the negative fractional derivative is defined by
\begin{equation}\label{RL2n}
{\mathbf D}^\alpha_{[x,\infty)} f(x):=\frac{\alpha(\alpha-1)}{\Gamma(2-\alpha)}\int_0^\infty \left[f(x+y)-f(x)-yf'(x)\right]y^{-1-\alpha}dy
\end{equation}
for $1<\alpha<2$ \cite[Eq.\ (3.33)]{FCbook}.  After a change of variables $y\mapsto -y$, it is not hard to see that these are special cases of the formula \eqref{PDO}.

The generator of any $\alpha$-stable semigroup on $\R$ with index $1<\alpha<2$ can be written as
\begin{equation}\label{GenEx1new2}
L f(x)=-af'(x)+\int_{y\neq 0} \left[f(x-y)-f(x)+yf'(x)\right]\phi(dy)
\end{equation}
with
\begin{equation}\label{LevyMeasStableDay10}
\phi(dy)=\begin{cases} \displaystyle{b\frac{\alpha(\alpha-1)}{\Gamma(2-\alpha)} y^{-\alpha-1}dy} &\text{for $y>0$, and}\\[12pt]
\displaystyle{c\frac{\alpha(\alpha-1)}{\Gamma(2-\alpha)} |y|^{-\alpha-1}dy}  &\text{for $y<0$.}
\end{cases}
\end{equation}
and a computation \cite[Example 3.24]{FCbook} shows that
\begin{equation}\label{FADE2}
Lf(x)=-a\partial_x f(x)+b {\mathbf D}^\alpha_{{(-\infty,x]}} f(x) +c {\mathbf D}^\alpha_{{[x,\infty)}}f(x) .
\end{equation}

The fractional partial differential equation $\partial_t u=Lu$ with generator \eqref{FADE2} is useful for modeling anomalous diffusion, where a cloud of particles spreads at a faster rate than a Brownian motion (the special case $\alpha=2$).  Figure \ref{MADEloglog} shows a typical application from Benson et al.\ \cite{Benson2001}.  The $\alpha$-stable densities that solve this fractional diffusion equation with $\alpha=1.1$, $a=0.12$ m/day, $b=0.14$ m$^\alpha$/day, and $c=0$ fit measured concentrations in an underground aquifer.  The best Gaussian solution gives a very poor fit on the leading tail, and hence significantly underestimates the risk of downstream contamination.

\begin{figure}[htb]
\begin{center}
  \includegraphics[width=5.0in]{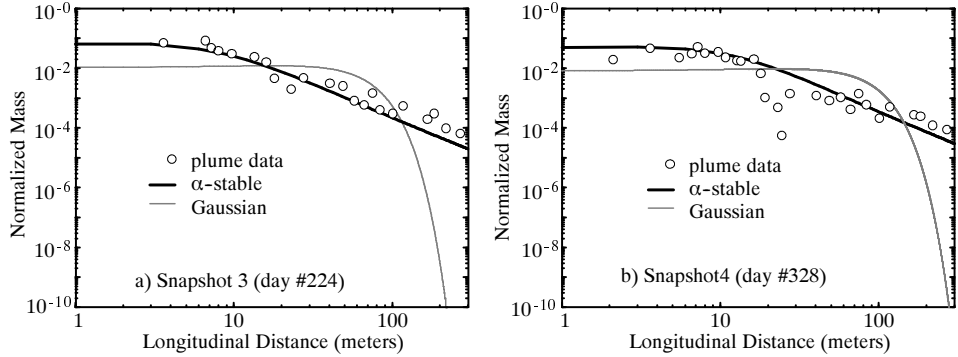}
\caption{Concentration measurements and fractional diffusion model \eqref{FADE2} from Benson et al.\ \cite{Benson2001}.   }
\label{MADEloglog}
  \end{center}
\end{figure}

An important open problem for fractional diffusion modeling is to identify the appropriate governing equation and boundary conditions on a bounded domain $\Omega=(L,R)$, see Defterli et al. \cite{FCAAnonlocal} for additional discussion.  The next result solves this problem in the case of zero Dirichlet boundary conditions.

\begin{theorem}\label{th:stable1d}
Assume that $X_t$ is a stable L\'evy process on $\R^1$ with generator \eqref{GenEx1new2} and L\'evy measure \eqref{LevyMeasStableDay10} for some $1<\alpha<2$. Let $\Omega=(L,R)$. Then the killed generator \eqref{GenDefOmega} is given by $L_\Omega f(x)=L^p f(x)$ for all $x\in\Omega$, for any $f\in C^2_0(\Omega )$ such that $L^pf(x)\in C_0(\Omega)$, where
\begin{equation}\begin{split}\label{GenEx1newOmega2}
L^p f(x)&=-a\partial_x f(x)+\int_{-\infty}^{\infty} \left[f(x-y)-f(x)+yf'(x)\right]\phi(dy)\\
&=-a\partial_x f(x)+b\,{\mathbb D}^\alpha_{[L,x]} f(x)+c\,{\mathbb D}^\alpha_{[x,R]} f(x) ,
\end{split}\end{equation}
using the Riemann-Liouville fractional derivatives \eqref{RLdef}.
The integral formula in \eqref{GenEx1newOmega2} is applied to the zero extension of $f\in C^2_0(\Omega )$, a function $f\in C_0(\R)$ defined by setting $f(x)=0$ for $x\notin\Omega$.
\end{theorem}

\begin{proof}
In order to apply the results of Section \ref{sec2}, we need to show that $\Omega$ is regular. For $x\in\RR$, define the first hitting time of $x$ by $T_x=\inf\left\{t>0:X_t=x\right\}$. Since $1<\alpha<2$, we have $\p^x(T_x=0)=1$ for all $x\in\RR$, see for example Sato \cite[Example 43.22, p. 325]{SAT}. This implies that the boundary points \TL{$L$ and $R$} are both regular for $\Omega$.  Since $X_t$ has a smooth density for any $t>0$ \cite[Theorem 7.2.7]{RVbook}, Remark \ref{RemHawkes} shows that Theorem \ref{masterCor} applies, and hence the killed generator is given by the formula \eqref{GenEx1new2} applied to the zero extension of a function $f\in C^2_0(\Omega )$.  For any such function, use \eqref{GenEx1new2} to write
\begin{equation}
L^p f(x)=-af'(x)+bI_1+cI_2
\end{equation}
where $$I_1=\int_0^\infty [f(x-y)-f(x)+yf'(x)]\frac{y^{-\alpha-1}}{\Gamma(-\alpha)}
\,dy$$ and $$I_2=\int_{-\infty}^0 [f(x-y)-f(x)+yf'(x)]\frac{|y|^{-\alpha-1}}{\Gamma(-\alpha)}
\,dy.$$
Write
\[I_1=\int_{x-L}^\infty [0-f(x)+yf'(x)]\frac{y^{-\alpha-1}}{\Gamma(-\alpha)}dy+\int_0^{x-L} [f(x-y)-f(x)+yf'(x)]\frac{y^{-\alpha-1}}{\Gamma(-\alpha)}dy\]
and integrate by parts, noting that  $f(x-y)-f(x)+yf'(x)=O(y^2)$ as $y\to 0$.  The remaining boundary terms from the two integrals cancel, and then a change of variable $y\mapsto x-y$ yields
\begin{equation}\begin{split}\label{I1calc}
I_1
=&f'(x)\frac{(x-L)^{1-\alpha}}{\Gamma(2-\alpha)}+\int_L^x [f'(y)-f'(x)]\frac{(x-y)^{-\alpha}}{\Gamma(1-\alpha)}\,dy.
\end{split}
\end{equation}
Write $D=d/dx$ and use \eqref{CaputoDef} to see that
\begin{equation*}\begin{split}
{\mathbb D}^\alpha_{[L,x]} f(x)&=D^2{\mathbb I}^{2-\alpha}_{[L,x]} f(x)
=D^1\left[D^1{\mathbb I}^{1-(\alpha-1)}_{[L,x]} f(x)\right]=D^1{\mathbb D}^{\alpha-1}_{[L,x]} f(x) .
\end{split}
\end{equation*}
The Caputo and Riemann-Liouville fractional derivatives are related by
\[
{\mathbb D}^{\alpha-1}_{[L,x]} f(x)=\partial^\alpha_{[L,x]} f(x)+f(L)\frac{(x-L)^{1-\alpha}}{\Gamma(2-\alpha)} ,
\]
see for example \cite[Eq.\ (2.4.6)]{Kilbas2006}.
Since $f(L)=0$, this implies that
\begin{equation*}\begin{split}
{\mathbb D}^\alpha_{[L,x]} f(x)&=D^1{\partial}^{\alpha-1}_{[L,x]} f(x)\\
&=\frac{d}{dx}\left[\int_L^{x} f'(y)\frac{(x-y)^{1-\alpha}}{\Gamma(2-\alpha)}\,dy\right]\\
&=\frac{d}{dx}\left[\int_L^{x} [f'(y)-f'(x)]\frac{(x-y)^{1-\alpha}}{\Gamma(2-\alpha)}\,dy\right]+\frac{d}{dx}\left[f'(x)\int_L^{x} \frac{(x-y)^{1-\alpha}}{\Gamma(2-\alpha)}\,dy\right]\\
&=\int_L^{x} [f'(y)-f'(x)]\frac{(x-y)^{-\alpha}}{\Gamma(1-\alpha)}\,dy+f'(x)\frac{d}{dx}\int_L^{x} \frac{(x-y)^{1-\alpha}}{\Gamma(2-\alpha)}\,dy ,
\end{split}
\end{equation*}
which reduces to \eqref{I1calc}.  Similarly, $I_2={\mathbb D}^\alpha_{[x,R]}f(x)$.
\end{proof}

Now the results stated in Section \ref{sec3} can be applied.   Suppose that $X_t$ is any stable process with index $1<\alpha<2$, specified by its generator \eqref{FADE2}.  \fe{Recall from Remark \ref{extension} that $L_\Omega$ is the unique continuous extension of \eqref{GenEx1newOmega2}.  In what follows, we will also denote this extension by $L_\Omega f(x)=-a\partial_x+b\,{\mathbb D}^\alpha_{[L,x]} f(x)+c\,{\mathbb D}^\alpha_{[x,R]} f(x)$.}  Then the function $u(x,t)=\Ex[f(X_{t})I\{\tau_\Omega<t\}]$ for $\Omega=(L,R)$ is the unique solution to the space-fractional Dirichlet problem
\begin{equation}\begin{split}\label{FADE2omega}
\partial_t u(x,t)&=-a\partial_x u(x,t)+b\, {\mathbb D}^\alpha_{[L,x]}u(x,t)+c\,{\mathbb D}^\alpha_{[x,R]} u(x,t)\quad \forall\  x\in\Omega,\ t>0\\
u(x,0)&=f(x)\quad \forall\ x\in \Omega;\\
u(x,t)&=0\quad \forall\ x\notin \Omega,t\geq 0
\end{split}\end{equation}
for any $f\in\D(L_\Omega)$, and the unique mild solution to \eqref{FADE2omega} for any $f\in C_0(\Omega)$.  If $u_1,u_2$ are the corresponding solutions to \eqref{FADE2omega} for initial functions $f_1,f_2$, then $\|u_2(t)-u_1(t)\|=\|P_t^\Omega(f_2-f_1)\|\leq \|f_2-f_1\|$ in the supremum norm, so the solution depends continuously on the initial condition.  Hence the Dirichlet problem \eqref{FADE2omega} is well posed.

Also, for any $0<\beta<1$ the function  $v(x,t)=\Ex[f(X_{E_t}^\Omega)]$ is the unique solution to the space-time fractional Dirichlet problem
\begin{equation}\begin{split}\label{tFADE2omega}
\partial_t^\beta v(x,t)&=-a\partial_x v(x,t)+b\, \BB{\mathbb D}^\alpha_{[l,x]}v(x,t)+c\,\BB{\mathbb D}^\alpha_{[x,r]} v(x,t)\quad \forall\  x\in\Omega,\ t>0\\
v(x,0)&=f(x)\quad \forall\ x\in \Omega;\\
v(x,t)&=0\quad \forall\ x\notin \Omega,t\geq 0
\end{split}\end{equation}
for any $f\in\D(L_\Omega)$, and the unique mild solution to \eqref{FADE2omega} for any $f\in C_0(\Omega)$.  Note that the process $X_{E_t}^\Omega$ is not Markov, and the family of operators $T_tf(x)=\Ex[f(X_{E_t}^\Omega)]$ is not a semigroup.  Write $v(x,t)$ in terms of $u(x,t)$ using \eqref{FCPsoln2}, where $h$ is given by \eqref{EtDens}.  Since $w\mapsto h(w,t)$ is the probability density function of the nonnegative random variable $E_t$, we have
\begin{equation}\begin{split}\label{tFADE2omegaNorm}
\|v(t)\|&=\sup_{x\in\Omega}\left|\int_0^\infty P_w^\Omega f(x)h(w,t)\,dw\right|\\
&\leq \int_0^\infty \|P_w^\Omega f\| h(w,t)\,dw\leq\|f\| \int_0^\infty  h(w,t)\,dw=\|f\|
\end{split}\end{equation}
using the fact that $\|P_t^\Omega f\|\leq \|f\|$ in the supremum norm on $C_0(\Omega)$.  It follows that the space-time fractional diffusion equation \eqref{tFADE2omega} is also well-posed.
\end{example}

\begin{example}\label{ex4}
The following is a typical example from numerical analysis, see for example \cite{2sided,CNfde}.
Consider the inhomogeneous fractional partial differential equation
\begin{equation}\label{FPDEex4}
\partial_t u(x,t) =b {\mathbb D}^\alpha_{[0,x]} u(x,t) +  c {\mathbb D}^\alpha_{[x,1]} u(x,t) + g(x,t)
\end{equation}
on a finite domain $\Omega=(0,1)$ with $1<\alpha<2$, positive coefficients $b\neq c$, initial condition $u(x,0) =0$ for all $x\in \Omega$, Dirichlet boundary conditions $u(0, t) = u(1,t) = 0$ for all $t\geq 0$, and forcing function
\begin{equation*}\begin{split}
g(x,t)&=x^2(1-x)^2-t\left[\frac{\Gamma(3)}{\Gamma(3-\alpha)}g_{2-\alpha}(x)-2\frac{\Gamma(4)}{\Gamma(4-\alpha)}g_{3-\alpha}(x)
                     +\frac{\Gamma(5)}{\Gamma(5-\alpha)}g_{4-\alpha}(x)\right]
\end{split}\end{equation*}
where $g_p(x)=ax^p+b(1-x)^p$.  Using the well-known formulae \cite[Example 2.7]{FCbook}
\begin{equation}\begin{split}\label{RLpowerlaw}
{\mathbb D}^\alpha_{[L,x]} (x-L)^p &= \frac {\Gamma(p+1)} {\Gamma(p+1-\alpha)}(x-L)^{p-\alpha}  \\
{\mathbb D}^\alpha_{[x,R]} (R-x)^p &= \frac {\Gamma(p+1)}{\Gamma(p+1-\alpha)} (R-x)^{p-\alpha}
\end{split}\end{equation}
for $p>\alpha$, it is easy to check that the exact solution is $u(x,t) = tx^2(1-x)^2$.  However, up to now, it was not known whether this solution was well-posed, or even unique, see \cite[Section 3]{FCAAnonlocal} for further details.

Since both $g(x,t)\in C_0(\Omega)$ and $\partial_t g(x,t)\in C_0(\Omega)$ for all $t\geq 0$, it follows from Example \ref{Ex1new} and \eqref{CPinhom} that this is the {\it unique} solution to the inhomogeneous Cauchy problem
\begin{equation}\begin{split}\label{FADE2omegaInhom}
\partial_t u(x,t)&=b\, {\mathbb D}^\alpha_{[0,x]}u(x,t)+c\,{\mathbb D}^\alpha_{[x,1]} u(x,t)+g(x,t)\quad \forall\  x\in\Omega,\ t>0\\
u(x,0)&=0\quad \forall\ x\in \Omega;\\
u(x,t)&=0\quad \forall\ x\notin \Omega,t\geq 0 .
\end{split}\end{equation}
Furthermore, uniqueness and \eqref{CpGenOmegaInhom} imply that $u(x,t)=\int_0^t \Ex[g(X_t^\Omega,t-s)]\,ds$.  Since the initial function $f(x)\equiv 0$, we certainly have $P_t^\Omega f\in C_0^2(\Omega)$ for all $t\geq 0$. Hence $u(x,t)$ is the unique solution to \eqref{FADE2omegaInhom} in the classical sense, \fe{i.e., the generator can be explicitly computed by the pointwise formulae \eqref{RLdef} for the Riemann-Liouville fractional derivatives}.  
\end{example}

\begin{remark}
An important question in the theory of fractional partial differential equations is how to write appropriate boundary conditions.  From the point of view of killed Markov processes, it is natural to impose the condition that $u(x,t)=0$ for all $ x\notin \Omega$ and all $t\geq 0$.  On the other hand, the problem \eqref{FPDEex4} only assumes $u(x, t) = 0$ for $x$ on the boundary of $\Omega$.  However, the problem \eqref{FPDEex4} as stated is indeed well-posed, because the definition of the Riemann-Liouville fractional derivative \eqref{RLdef} implicitly incorporates the zero exterior condition.
\end{remark}

\begin{remark}
In some applications, the Caputo fractional derivatives \eqref{CaputoDef} in the spatial variable $x$ are used instead of the Riemann-Liouville.  For the problem \eqref{FPDEex4}, these two forms are equivalent, because both $u(x,t)$ and $\partial_x u(x,t)$ vanish at the boundary, see for example Podlubny \cite[Eq.\ (2.165)]{Podlubny}.
\end{remark}

\TL{
\begin{remark}\label{RegularityStable}
The generator of an $\alpha$-stable L\'evy process $X_t$ on $\R$ with index $0<\alpha<1$ can be written in the form
\begin{equation}\label{GenEx1new1}
L f(x)=-af'(x)+\int_{y\neq 0} \left[f(x-y)-f(x)\right]\phi(dy)
\end{equation}
where
\begin{equation}\label{LevyMeasStable1}
\phi(dy)=\begin{cases} \displaystyle{b\,\frac{\alpha}{\Gamma(1-\alpha)} y^{-\alpha-1}dy} &\text{for $y>0$, and}\\[12pt]
\displaystyle{c\,\frac{\alpha}{\Gamma(1-\alpha)} |y|^{-\alpha-1}dy}  &\text{for $y<0$.}
\end{cases}
\end{equation}
Using the {generator form} of the positive fractional derivative \cite[Eq.\ (2.15)]{FCbook}
\begin{equation}\label{RL1p}
{\mathbf D}^\alpha_{[-\infty,x]} f(x):=\frac{\alpha}{\Gamma(1-\alpha)}\int_0^\infty \left[f(x)-f(x-y)\right]y^{-1-\alpha}dy
\end{equation}
and the negative fractional derivative \cite[Eq.\ (3.31)]{FCbook}
\begin{equation}\label{RL1n}
{\mathbf D}^\alpha_{[x,\infty]} f(x):=\frac{\alpha}{\Gamma(1-\alpha)}\int_0^\infty \left[f(x)-f(x+y)\right]y^{-1-\alpha}dy
\end{equation}
for $0<\alpha<1$, we can also write
\begin{equation}\label{FADE1}
Lf(x)=-a\partial_x f(x)-b {\mathbf D}^\alpha_{{[-\infty,x]}} f(x) -c {\mathbf D}^\alpha_{{[x,\infty]}}f(x) ,
\end{equation}
see \cite[Example 3.24]{FCbook} for details.
The question whether $\Omega=(L,R)\subset\RR$ is regular can be answered in terms of the first passage time of $X_t$, which is defined by
$$
T_{(x,\infty)}=\inf\left\{t>0:X_t>x\right\},\quad x\in\RR.
$$
Since $X_t$ is continuous in probability, it follows that $R$ is regular for $\Omega$ if and only if ${\mathbb P}^R(T_{(R,\infty)}=0)=1$, and the regularity of $L$ can be described analogously in terms of $T_{(-\infty,L)}$.  It follows using \cite[Theorem 47.6]{SAT}  that $\Omega$ is
regular if and only if $b>0$, $c>0$ and $a=0$.  Then an argument similar to Theorem \ref{th:stable1d} shows that
the generator of the killed stable L\'evy process is given by
\begin{equation}\begin{split}\label{GenEx1newOmega1}
L_\Omega f(x)&=-b\,\partial^\alpha_{[L,x]} f(x)-c\,\partial^\alpha_{[x,R]} f(x)
\end{split}\end{equation}
for all $x\in\Omega$, for any $f\in C^2_0(\Omega )$ such that the right-hand side of \eqref{GenEx1newOmega1} belongs to $C_0(\Omega)$.
It also follows from \cite[Theorem 47.6]{SAT}  that $\Omega$ is always regular for $X_t$ when $\alpha=1$.  One can also compute the generator of the corresponding killed process on $\Omega$, but the formula is more complicated, because the centering term $f'(x)yI_{B_1}(y)$ in \eqref{PDO} cannot be simplified.
\end{remark}
}

\begin{remark}\label{HarishRemark}
Suppose that $c=0$ and $1<\alpha<2$ in \FE{\eqref{GenEx1newOmega2}}.  Then Theorem 3.4.4 and Theorem 4.3.3  in the recent PhD thesis of Sankaranarayanan \cite{Sankaranarayanan2014} show that the domain of the killed generator $L_\Omega$ for $\Omega=(0,1)$ can be characterized completely as
$$D(L_\Omega)=\left\{f\in C_0(\Omega):f={\mathbb I}_{[0,x]}^{\alpha}g-x^{\alpha-1}{\mathbb I}_{[0,x]}^{\alpha}g(1)\ \exists\ g\in C_0(\Omega)\right\} .$$
The second term $x^{\alpha-1}{\mathbb I}_{[0,x]}^{\alpha}g(1)$ ensures that $f(1)=0$.  Then $L_\Omega f\, \fe{=b{\mathbb D}^\alpha_{[0,x]}f}=g$, since \fe{${\mathbb D}^\alpha_{[0,x]}{\mathbb I}^\alpha_{[0,x]}f=f$ for all $f\in C_0(\Omega)$ \cite[Eq.\ (2.106)]{Podlubny}, and}
\[L_\Omega[x^{\alpha-1}]=bD^2{\mathbb I}^{2-\alpha}_{[0,x]}x^{\alpha-1}=bD^2\left[\frac{\Gamma(\alpha)}{\Gamma(2-\alpha)} x\right]=0\]
for all $x\in(0,1)$.  Hence the point-wise formula \eqref{GenEx1newOmega2} for $L_\Omega f(x)$ is valid for all $f\in D(L_\Omega)$ in this case.
\end{remark}

\begin{example}
Meerschaert and Tadjeran \cite{2sided} consider
\begin{equation}\label{FPDE}
\partial_t u(x,t) =a(x) {\mathbb D}^{1.8}_{[0,x]} u(x,t) +  b(x) {\mathbb D}^{1.8}_{[x,2]} u(x,t) + g(x,t)
\end{equation}
on a finite domain $0 < x < 2$ and $t > 0$ with the coefficient functions
\[a(x) = \Gamma(1.2) \; x^{1.8} \quad\text{and}
\quad b(x) = \Gamma(1.2) \; (2-x)^{1.8} ,\]
the forcing function
\[
g(x,t) = - 32 e^{-t}[x^2+(2-x)^2-2.5\bigl(x^3+(2-x)^3 \bigr) +\;\;\frac{25}{22}(x^4+(2-x)^4)] ,
\]
initial condition $u(x,0) = 4 x^2 (2-x)^2$, and Dirichlet boundary conditions $u(0, t) = u(2,t) = 0$.
Using \ref{RLpowerlaw}, is easy to check that $u(x,t) = 4 e^{-t}x^2(2-x)^2$ is the exact solution.  This test problem is used in \cite{2sided} to demonstrate the effectiveness of an implicit Euler solution method.  The method is proven to be unconditionally stable and consistent, and hence convergent, but whether the problem is well-posed is an open question, see Defterli et al.\ \cite{FCAAnonlocal} for additional discussion.  The operator $L=a(x,t) {\mathbb D}^{1.8}_{[-\infty,x]}  +  b(x,t) {\mathbb D}^{1.8}_{[x,\infty]} $ can be computed from \eqref{PDO} with $c=l=Q=0$ and
\begin{equation}\label{LevyMeasEx5}
N(x,dy)=c(x,y)\frac{\alpha(\alpha-1)}{\Gamma(2-\alpha)}|y|^{-\alpha-1}dy, \quad c(x,y)=b(x)I(y>0)+a(x)I(y<0).
\end{equation}
However, it is not known whether this stable-like operator generates a Markov process on $\R$.  In particular, the coefficients do not satisfy the usual growth conditions for a stochastic differential equation, see \cite[Theorem A.1]{Chakraborty}.  We can, however, prove uniqueness using the following well-known result.

\begin{proposition}\label{prop:unique}
{Suppose that $\Omega$ is a bounded domain in $\Rd$, and $F(r)\geq F(s)$ for $r\leq s$.  Define the operator $If(x)=F(Lf(x))$ where $Lf(x)$ is given by \eqref{PDO}.  If $u,v$ are two solutions to
\begin{equation}\begin{split}\label{e:4.3inhom}
\partial_t u(x,t) +I u(x,t)&= 0; \quad  x\in \Omega, \ 0<t<T\\
u(x,t)&=h(t,x), \quad  x\notin  \Omega, \ 0<t<T, \\
u(x,0)& = f(x), \quad x\in \Omega
\end{split}\end{equation}
for some $T>0$, then $u(x,t)=v(x,t)$ for all $x\in \Rd$ and all $t\geq 0$.}
\end{proposition}

\begin{proof}
{[Thanks to Andrzej Swiech] Suppose that $u(y,s)>v(y,s)$ at some point $y\in \Omega$ and $0<s< T$.  For $\delta>0$, define
\[u^\delta(x,t):=u(x,t)-\frac{\delta}{T-t} .\]
If $\delta>0$ is sufficiently small, then $u^\delta(y,s)-v(y,s)>0$, and hence the function $u^\delta(x,t)-v(x,t)$ attains its positive maximum at some point $(x,t)\in \Omega\times(0,T)$.  Then at this point we have {$\partial_t u^\delta(x,t)=\partial_t v(x,t)$, and $\nabla u^\delta(x,t)=\nabla v(x,t)$}.  Since $u^\delta(x+z,t)-v(x+z,t)\leq u^\delta(x,t)-v(x,t)$, we also have $u^\delta(x+z,t)-u^\delta(x,t)\leq v(x+z,t)-v(x,t)$, and it follows that $Lu^\delta(x,t)\leq  Lv(x,t)$.  Hence $Iu^\delta(x,t)\geq Iv(x,t)$.  Thus we obtain
\[0=\partial_t v(x,t) +I v(x,t)\leq \partial_t u^\delta(x,t) +I u^\delta(x,t)=\frac{-\delta}{(T-t)^2} \]
which is a contradiction.
}\end{proof}

\fe{Since we know that $u(x,t) = 4 e^{-t}x^2(2-x)^2$ solves the Dirichlet problem \eqref{FPDE}, we can apply Proposition \ref{prop:unique} with $F(u)=-u$ to show that this solution is unique.  Hence the numerical method in \cite{2sided} indeed converges to the unique solution, which resolves an open question in that paper.}
\end{example}

\begin{example}\label{ex2new}
The generator of any $\alpha$-stable semigroup on $\Rd$ with index $1<\alpha<2$ can be written in the form
\begin{equation}\label{GenEx2new2}
L f(x)=-a\nabla f(x)+\int_{y\neq 0} \left[f(x-y)-f(x)+y\cdot\nabla f(x)\right]\phi(dy)
\end{equation}
where
\begin{equation}\label{LevyMeasStableRd2}
\phi(dy)=b\,\frac{\alpha(\alpha-1)}{\Gamma(2-\alpha)} r^{-\alpha-1}dr\,M(d\theta)
\end{equation}
in polar coordinates $r=|y|$ and $\theta=y/|y|$, where the spectral measure $M(d\theta)$ is any probability measure on the unit sphere. A calculation shows that
\begin{equation}\label{FADE2Rd}
Lf(x)=-a\nabla f(x)+ b\nabla^\alpha_M f(x) ,
\end{equation}
where the vector fractional derivative is defined by
\[\nabla^\alpha_M f(x)=\int_{|\theta|=1} {\mathbf D}_\theta^\alpha f(x)\,M(d\theta)\]
and ${\mathbf D}_\theta^\alpha$ the fractional directional derivative, i.e., the one dimensional fractional derivative ${\mathbf D}^\alpha_r g(r)$ (in generator form) of the projection $g(r)=f(x+r\theta)$ for $r\in \R$.  See \cite[Example 6.29]{FCbook} for complete details.

If $M$ is uniform over the sphere, it follows that \BB{$\nabla^\alpha_M f(x)=-c_{d,\alpha}(-\Delta)^{\alpha/2}f(x)$}, where the fractional Laplacian \BB{$-(-\Delta)^{\alpha/2}f(x)$} has Fourier transform $-\|k\|^\alpha \hat f(k)$, and \BB{
\[c_{d,\alpha}=|\cos(\pi\alpha/2)|\int_{\|\theta\|=1} |\theta_1|^\alpha M(d\theta)\]}
where $\theta=(\theta_1,\ldots,\theta_d)$, see \cite[Example 6.24]{FCbook}.

For any stable L\'evy process with index $1<\alpha<2$, Remark \ref{RemHawkes} shows that Theorem \ref{masterCor} applies for any regular bounded domain $\Omega\subset\Rd$, and hence the killed generator is given by the same formula \eqref{GenEx2new2} applied to the zero extension a function $f\in C^2_0(\Omega)$.  Now suppose that $\Omega$ is a \fe{convex} domain, so that for every $x\in\Omega$ and $|\theta|=1$ there exists a unique $R=R(x,\theta)>0$ such that $x-r\theta\in\Omega$ for $0<r<R$, and $x-r\theta\notin\Omega$ for $r>R$. Let $C=b \alpha(\alpha-1)/\Gamma(2-\alpha)$.  A change of variable $y=r\theta$ in polar coordinates yields
\[
L_\Omega f(x)=-a\nabla f(x)+\int_{|\theta|=1} \int_0^{\infty}\left[f(x-r\theta)-f(x)+r\theta\cdot\nabla f(x)\right]C r^{-\alpha-1}dr\,M(d\theta)
\]
for any $f\in C^2_0(\Omega)$ such that the right-hand side belongs to $C_0(\Omega)$.  Then the same one dimensional calculation on the inner integral as in Example \ref{Ex1new} leads to
\begin{equation}\label{FADE2RdOmega}
L_\Omega f(x)=-a\nabla f(x)+ b\nabla^\alpha_{M,\Omega} f(x)
\end{equation}
where
\begin{equation}\label{Ex2new2FracDvtOmega}
\nabla^\alpha_{M,\Omega} f(x)=\int_{|\theta|=1} {\mathbb D}^\alpha_{[x-R(x,\theta),x],\theta}f(x)\, M(d\theta) ,
\end{equation}
and  ${\mathbb D}^\alpha_{[x-R,x],\theta}f(x)$ is the Riemann-Liouville fractional directional derivative, defined as the one dimensional Riemann-Liouville derivative $\partial^\alpha_{[x-R,x]}g(r)$ of the projection $g(r)=f(x+r\theta)$.  Note that $g'(r)=\theta\cdot\nabla f(x+r\theta)$.

Then for any $0<\beta< 1$ the function  $v(x,t)=\Ex[f(X_{E_t}^\Omega)]$ is the unique solution to the Dirichlet problem
\begin{equation}\begin{split}\label{ex2newTFADE}
\partial_t^\beta v(x,t)&=-a\partial_x v(x,t)+b\nabla^\alpha_{M,\Omega}v(x,t)\quad \forall\  x\in\Omega,\ t>0\\
v(x,0)&=f(x)\quad \forall\ x\in D;\\
v(x,t)&=0\quad \forall\ x\notin D,t\geq 0
\end{split}\end{equation}
for any $f\in\D(L_\Omega)$, and the unique mild solution to \eqref{FADE2omega} for any $f\in C_0(\Omega)$. Then the same argument as in Example \ref{Ex1new} shows that the space-time fractional diffusion equation \eqref{ex2newTFADE} is well-posed.  \fe{As in the previous examples, we understand that \eqref{FADE2RdOmega} represents the unique continuous extension to $\D(L_\Omega)$.}
\end{example}

\begin{remark}
Example \ref{ex2new} includes the fractional Laplacian as a special case.  Chen et al.\ \cite[Theorem 5.1]{CMN} established strong solutions to the space-time fractional diffusion equation with Dirichlet boundary conditions \eqref{ex2newTFADE} in the special case where $M(d\theta)$ is uniform over the sphere, i.e., the fractional Laplacian.
Here the function $u(x,t)$ is said to be a strong solution if for every $t>0$, $u(x,t)\in C_0(\Omega)$, \BB{$(-\Delta)^{\alpha/2} u(x,t)$} exists pointwise for every
$x\in \Omega$, the Caputo fractional derivative $\partial^\beta_t u(x,t)$ exists pointwise for every $t>0$ and $x\in \Omega$, \BB{$\partial^\beta_t u(x,t) =
-(-\Delta)^{\alpha/2} u(x,t)$} pointwise in $(0, \infty)\times \Omega$, and $\lim_{t\downarrow 0} u(x,t) =f(x)$ for every $x\in \Omega$.  The theorem assumes that the initial condition $f\in \D(L_\Omega^k)$ for some  $k>-1+({3d+4})/({2\alpha})$.
The proof of \cite[Theorem 5.1]{CMN} involves symmetric Dirichlet forms, and an eigenfunction expansion of the fractional Laplacian.  It seems difficult to extend that argument to the more general setting of Example \ref{ex2new}, since the generator $L$ of a stable process need not be \fe{self-adjoint}, so that standard spectral theory does not apply.
\end{remark}

\begin{example}\label{ex2}
Bass \cite{Bass88} introduced stable-like processes, where the order $\alpha(x)$ of the fractional derivative varies in space.  If $\alpha:\Omega\to [\alpha_1,\alpha_2]$ is a smooth bounded function for some $0<\alpha_1<\alpha_2<2$, then Schilling and Wang \cite[Theorem 3.3]{SW12} prove that the stable-like process $X_t$ on $\rd$ with generator \BB{$-(-\Delta)^{\alpha(x)/2}$} is doubly Feller.  If $\Omega$ is a regular bounded domain in $\Rd$, then Bass \cite[Theorem 2.1 and Remark 7.1]{Bass1988} shows that $X_t$ solves the martingale problem, i.e.,
\[f(X_t)-f(X_0)-\int_0^t Lf(X_s)\,ds\]
is a $\sigma\{X_s:0\leq s\leq t\}$-martingale for any $f\in C^2_b(\rd)$, the family of real-valued functions on $\rd$ such that $f$ and all its derivatives of order 1 or 2 are continuous and bounded.  Then it is easy to check, using the definition of the generator, that any function $f\in C_0^2(\Rd)$ is in $\D(L)$, where $L$ is given by \eqref{PDO} with  $c=l=Q=0$ and $N(x,dy)=c_{d,\alpha(x)}|y|^{-d-\alpha(x)}dy$ for any $f\in C_0^2(\Rd)$.  Then Theorem \ref{masterCor} shows that the generator of the killed process is given by the same pointwise formula applied to the zero extension of a function $f\in C_0^2(\Omega)$.  Then for any $0<\beta< 1$ the function  $v(x,t)=\Ex[f(X_{E_t}^\Omega)]$ is the unique solution to the Dirichlet problem
\begin{equation}\begin{split}\label{ex2TFADE}
\partial_t^\beta v(x,t)&=\BB{-(-\Delta)^{\alpha(x)/2}}v(x,t)\quad \forall\  x\in\Omega,\ t>0\\
v(x,0)&=f(x)\quad \forall\ x\in \Omega;\\
v(x,t)&=0\quad \forall\ x\notin \Omega,t\geq 0
\end{split}\end{equation}
for any $f\in\D(L_\Omega)$, and the unique mild solution to \eqref{ex2TFADE} for any $f\in C_0(\Omega)$. The same argument as in Example \ref{Ex1new} shows that the Dirichlet problem \eqref{ex2TFADE} is well-posed.  \fe{Here again, we define $-(-\Delta)^{\alpha(x)/2}f(x)$ using the zero extension of a function $f\in C_0(\Omega)$, and we have $f\in\D(L_\Omega)$ if the pointwise formula for $-(-\Delta)^{\alpha(x)/2}f(x)$ belongs to $C_0(\Omega)$. }
\end{example}

\begin{example}\label{ex3}
\mmm{Bass and Levin \cite{BL02} consider a different class of stable-like processes on $\rd$ with generator \eqref{PDO} where $c=l=Q=0$ and $N(x,dy)=\kappa(x,y)|y|^{-d-\alpha}dy$,  $0<\alpha<2$, $\kappa(x,y)=\kappa(x,-y)$, and $0<\kappa_1<\kappa(x,y)<\kappa_2<\infty$.  Here we assume that $\kappa(x,y)=a(x)c_{d,\alpha}$ where  $|a(x)-a(y)|\leq a_0 |x-y|^\lambda$ for some $0<\lambda<1$ and $a_0>0$.  Theorem 3.19 in B\"ottcher et al.\ \cite{BSW} establishes the existence of a time-homogeneous Feller process $X_t$ with this generator \BB{$L=-a(x) (-\Delta)^{\alpha/2}$}.  Chen and Zhang \cite[Eq. (1.18)]{CZ13} observe that $X_t$ solves the stochastic differential equation $dX_t=a(X_{t-})^{1/\alpha}dY_t$ where $Y_t$ is the standard symmetric stable L\'evy process with generator \BB{$L_Y=-(-\Delta)^{\alpha/2}$} for some $0<\alpha<2$.   It follows from \cite[Corollary 1.3]{CZ13} that the transition density $p_t(x,y)$ of $X_t$ (i.e., the Lebesgue probability density of $y=X_{t+s}$ given $X_s=x$) is locally bounded in $(x,y)\in\rd\times\rd$ for any $t>0$.}  It is easy to check that $T_t$ is a $C_b(\rd)$ semigroup (e.g., see discussion after \cite[Theorem 2.1]{SW12}) and then it follows from Schilling and Wang \cite[Corollary 2.2]{SW12} that $X_t$ is doubly Feller.  Then for any regular bounded domain $\Omega\subset \Rd$, Theorem \ref{masterCor} shows that the generator of the killed process $X_t^\Omega$ is given by the same formula: \fe{$L_\Omega f(x)=-a(x) (-\Delta)^{\alpha/2}f(x)$ for all $f\in C^2_0(\Omega)$ such that $-a(x) (-\Delta)^{\alpha/2}f(x)\in C_0(\Omega)$, where we define $f(x)=0$ for $x\notin\Omega$}.  Hence for any $0<\beta< 1$ the function  $v(x,t)=\Ex[f(X_{E_t}^\Omega)]$ is the unique solution to the Dirichlet problem
\begin{equation}\begin{split}\label{ex3TFADE}
\partial_t^\beta v(x,t)&=-a(x) \BB{(-\Delta)^{\alpha/2}}v(x,t)\quad \forall\  x\in\Omega,\ t>0\\
v(x,0)&=f(x)\quad \forall\ x\in \Omega;\\
v(x,t)&=0\quad \forall\ x\notin \Omega,t\geq 0
\end{split}\end{equation}
for any $f\in\D(L_\Omega)$, and the unique mild solution to \eqref{ex3TFADE} for any $f\in C_0(\Omega)$. The same argument as in Example \ref{Ex1new} shows that the Dirichlet problem \eqref{ex3TFADE} is well-posed.  \fe{Again, $-a(x) (-\Delta)^{\alpha/2}$ represents the unique continuous extension to $\D(L_\Omega)$, and $f\in\D(L_\Omega)$ if the pointwise formula for $-a(x) (-\Delta)^{\alpha/2}f(x)$ belongs to $C_0(\Omega)$. }
\end{example}

\begin{acknowledgment}
The authors would like to thank Krzysztof Bogdan, Wroc\l{}aw University of Technology, Zhen-Qing Chen, University of Washington, Ren\'e Schilling, Technical University of Dresden, and Andrzej Swiech, Georgia Institute of Technology, for helpful discussions. B.~Baeumer was supported by the Marsden Fund Council from Government funding, administered by the Royal Society of New Zealand.  M. M. Meerschaert was partially supported by ARO grant W911NF-15-1-0562 and NSF grants EAR-1344280 and DMS-1462156.
\end{acknowledgment}

\bibliographystyle{plain}

\end{document}